\documentclass[10pt]{amsart}

\usepackage{amsmath}
\usepackage{amssymb}
\usepackage{amsfonts}
\usepackage{geometry}
\usepackage{setspace}
\usepackage{url}
\usepackage{color}
\usepackage{blkarray}
\usepackage{comment}
\usepackage{blindtext,rotating}
\usepackage{bbm}
\usepackage{tikz-cd}
\usepackage{setspace}
\usepackage{amsthm}
\usepackage[all,cmtip]{xy}
\usepackage{amscd}

\usepackage{amsthm}
\theoremstyle{definition}
\newtheorem{theorem}{Theorem}[section]
\newtheorem{lemma}[theorem]{Lemma}

\newtheorem{corollary}[theorem]{Corollary}

\newtheorem{definition}[theorem]{Definition}

\newtheorem{question}[theorem]{Question}

\newtheorem{notation}[theorem]{Notation}

\newcommand{\bs}{\mathbf{s}}

\newcommand{\bla}{\boldsymbol\lambda}

\newcommand{\C}{\mathbb{C}}
\newcommand{\Z}{\mathbb{Z}}

\newcommand{\cB}{\mathcal{B}}
\newcommand{\cC}{\mathcal{C}}
\newcommand{\cD}{\mathcal{D}}
\newcommand{\cF}{\mathcal{F}}
\newcommand{\cO}{\mathcal{O}}

\newcommand{\KZ}{\mathsf{KZ}}

\newcommand{\E}{\mathrm{E}}
\newcommand{\F}{\mathrm{F}}

\newcommand{\Hom}{\operatorname{Hom}}
\newcommand{\End}{\operatorname{End}}

\newcommand{\Ext}{\operatorname{Ext}}
\newcommand{\Ind}{\operatorname{Ind}}

\newcommand{\Res}{\operatorname{Res}}

\newcommand{\Irr}{\operatorname{Irr}}

\newcommand{\Id}{\mathrm{Id}}
\newcommand{\Tr}{\mathrm{Tr}}

\newcommand{\sle}{\widehat{\mathfrak{sl}}_e}

\newcommand{\onto}{\twoheadrightarrow}

\begin{document}

\title{The Dipper-Du Conjecture revisited}

\author{Emily Norton}

\keywords{Hecke algebra, rational Cherednik algebra, type A, Dipper-Du conjecture, vertices and sources, Harish-Chandra theory, cuspidal supports}

\subjclass[2010]{16G99, 20C08, 20C30}

\begin{abstract}
We consider vertices, a notion originating in local representation theory of finite groups, for the category $\cO$ of a rational Cherednik algebra and prove the analogue of the Dipper-Du Conjecture for Hecke algebras of symmetric groups in that setting. As a corollary we obtain a new proof of the Dipper-Du Conjecture over $\mathbb{C}$. 
\end{abstract}

\maketitle

\section*{Introduction}
Let $H_q(S_n)$ be the Hecke algebra of the symmetric group with $q$ a primitive $e$-th root of $1$ and let $H_q(S_n)-$mod be the category of finite-dimensional $H_q(S_n)-$modules. If $M\in H_q(S_n)-$mod, a parabolic subgroup $S_\mu\subseteq S_n$ is called a \textit{vertex} of $M$ if $S_\mu$ is minimal with respect to the property that $M$ is isomorphic to a direct summand of a module induced from $H_q(S_\mu)$. The Dipper-Du Conjecture in characteristic $0$ states  that the parabolics of $S_n$ occurring as vertices of indecomposable modules in $H_q(S_n)-$mod are exactly the parabolics isomorphic to $S_e^{\times k}$, $0\leq k\leq \lfloor \frac{n}{e}\rfloor$ \cite{DipperDu}.  The conjecture was first proved by Du by demonstrating the invertibility of a certain norm map on the Hecke algebra \cite{Du}. The complete version of the conjecture over a ground field of characteristic $p\geq 0$, where ``$e$-parabolics" $S_e^{\times k}$ are supplemented by additional ``$e$-$p$-parabolics" when $p>0$, was recently proved by Whitley who defined and computed the vertices of the blocks of $H_q(S_n)$ as bimodules \cite{Whitley}.

When the ground field is $\C$, the quotient functor $\KZ:\cO_c(S_n)\rightarrow H_q(S_n)-$mod from the category $\cO_c(S_n)$ of the rational Cherednik algebra at parameter $c=r/e$, such that $q=\exp(2\pi i c)$, outfits these two categories with a means of passing information back and forth \cite{GGOR}.
A theorem of Wilcox identifies the cuspidal supports of all simple modules in $\cO_c(S_n)$ as the parabolics $S_e^{\times k}$ for $0\leq k\leq \lfloor \frac{n}{e}\rfloor$ -- the same answer as for the vertices of the Hecke algebra \cite{Wilcox}.\footnote{See also \cite{ShanVasserot}.} 
Motivated by this striking coincidence, we look at vertices for the category $\cO_c(S_n)$ of the Cherednik algebra and establish the analogous statement to Dipper-Du's conjecture in that setting (Theorem \ref{VerticesCherednik}). As a corollary, we obtain a new proof of the Dipper-Du Conjecture for the Hecke algebra over $\C$ (Theorem \ref{VerticesHecke}). We identify the vertex of a block in $\cO_c(S_n)$ using the simple modules in the block of minimal cuspidal depth; although the $\KZ$ functor kills these modules, it preserves the vertex of the block via their projective covers. 

We would like to raise the question of what happens if $S_n$ is replaced by an arbitrary complex reflection group $W$: does it remain true that the set of vertices of $H_q(W)$ coincides with the set of parabolic subgroups $\underline{W}\subseteq W$ such that $\cO_c(\underline{W})$ contains a cuspidal simple module? We always have inclusion in one direction: if $L\in\cO_c(W)$ is a simple module such that ${^\cO\Res}^{W}_{\underline{W}}L$ is cuspidal, then $\underline{W}$ is the vertex of the projective cover $P$ of $L$ \cite{GriffethJuteau}. Moreover, the vertex of $P$ is the vertex of $\KZ(P)$ \cite{GriffethJuteau}. Thus projective indecomposable modules in Cherednik category $\cO$ provide a wealth of vertices for Hecke algebras. For instance, combined with Shan-Vasserot's characterization of cuspidal supports for simple modules in $\cO_c(G(\ell,1,n))$ using categorical actions  \cite[Lemma 6.1]{ShanVasserot}, this implies the following observation:
\textit{
If $|\bla,\bs\rangle\in\cF_{e,\bs}$ is killed by the annihilation operators for the Heisenberg and $\sle$ crystals and $|\bla|\leq n$ then for each $0\leq k\leq \lfloor \frac{n-|\bla|}{e}\rfloor$, the parabolic subgroup $G(\ell,1,|\bla|)\times S_e^{\times k}$ of $ G(\ell,1,n)$ is the vertex of a projective indecomposable module $P\in\cO_c(G(\ell,1,n))$ and of $\KZ(P)\in H_q(G(\ell,1,n))-$mod.}
 Here $\bla=(\lambda^1,\dots,\lambda^\ell)$ where $\lambda^j$ are partitions and $|\bla|=\sum_{j=1}^\ell|\lambda^j|$, $\cF_{e,\bs}$ is a level $\ell$ Fock space of rank $e\in\Z_{\geq 2}$ and charge $\bs\in\Z^\ell$, and the parameters $c$ and $q$ are determined from $e$ and $\bs$, see e.g. \cite{GeckJacon}, \cite{ShanVasserot}.
\begin{question}
Let $W$ be a complex reflection group. Is the set of vertices of projective indecomposable modules in $\cO_c(W)$ a complete  set of vertices for $\cO_c(W)$ and $H_q(W)-$mod?
\end{question}


\section{Adjunctions}

We refer to \cite{MacLane} for all category-theoretic notions.
Let $A$ and $B$ be finite-dimensional algebras over a field $k$, and let $\cC=A-$mod and $\cD=B-$mod be the categories of finitely generated left $A-$ and $B-$ modules, respectively. For this section, we suppose we are given exact, biadjoint functors $\E:\cC\rightarrow\cD$ and $\F:\cD\rightarrow\cC$. 
The biadjunction yields a natural transformation of the identity functor on $\cC$:
$$\zeta:\mathbbm{1}_\mathcal{C}\stackrel{\eta}{\longrightarrow}\F\E\stackrel{\varepsilon}{\longrightarrow}\mathbbm{1}_\mathcal{C}$$
where $\eta$ is the unit of the adjunction $(\E,\F)$ and $\varepsilon$ is the counit of the adjunction $(\F,\E)$ . Write $\eta_M$, $\varepsilon_M$, $\zeta_M$ for the components of $\eta$, $\varepsilon$, $\zeta=\varepsilon\eta$ at the object $M\in\cC$.

Recall that $\cC$ has a direct sum decomposition into \textit{blocks}, which are the module categories of the indecomposable direct factors of $A$ as a $k$-algebra.
\begin{lemma}\label{simple}
Suppose $L$ and $L'$ are simple modules in the same block $\cB$ of $\cC$. Then $\zeta_L$ is an isomorphism if and only if $\zeta_{L'}$ is an isomorphism.
\end{lemma}
\begin{proof} Simples $L$ and $L'$ are in the same block $\cB$ if and only if there exist simples $L_1:=L,\;L_2,\;\dots,\;L_{r-1},\;L_r:=L'$ such that $\Ext^1(L_i,L_{i+1})\neq 0$ for all $i=1,\dots,r-1$ \cite[Proposition 13.3]{Alperin}. It therefore suffices to show that given a nonsplit short exact sequence
$$0\longrightarrow L'\stackrel{\iota}{\longrightarrow} M\stackrel{\pi}{\longrightarrow} L\longrightarrow 0$$
with $L,L'$ simple, $\zeta_L$ is an isomorphism if and only if $\zeta_{L'}$ is an isomorphism. We have the following commutative diagram 
whose top and bottom rows are exact:
\begin{equation}\label{natltrans}
 \begin{tikzcd}
0 \arrow{r}{}
& 
 L' \arrow{r}{\iota} \arrow{d}{\zeta_{L'}}
 &
 M \arrow{d}{\zeta_M} \arrow{r}{\pi}
 &
 L \arrow{d}{\zeta_L} \arrow{r}{}
 &
 0
 \\
 0\arrow{r}{}
 &
 L' \arrow{r}{\iota}
 &
 M \arrow{r}{\pi}
 &
 L \arrow{r}{}
 &
 0
 \end{tikzcd}
 \end{equation}
By assumption $M$ is indecomposable, so $\End(M)$ is a local ring, and therefore every element of $\End(M)$ is either nilpotent or invertible. If $\zeta_M$ is nilpotent, then taking $n$ such that $\zeta_M^n=0$, the diagram 
\begin{equation}\label{local}
\begin{tikzcd}
M\arrow{r}{\pi} \arrow{d}{\zeta_M^n=0}
&
L\arrow{d}{\zeta_L^n}
\\
M\arrow{r}{\pi}
&
L
\end{tikzcd}\end{equation}
must commute. But $\zeta_L^n$ is an isomorphism since $\zeta_L$ is, and so $(\zeta_L^n)\pi$ is surjective, while $\pi\zeta^n_M=0$. This is a contradiction, so $\zeta_M$ is an invertible element of $\End(M)$, that is, $\zeta_M:M\rightarrow M$ is an isomorphism. It then follows from the Five Lemma that $\zeta_{L'}$ is also an isomorphism. The converse implication, that $\zeta_L$ is an isomorphism if $\zeta_{L'}$ is, is proved similarly.
\end{proof}
\begin{notation} As in \cite[Section 6.B]{Broue}, if $M,X\in\cC$ and there exist morphisms $\iota:M\rightarrow X$ and $\pi:X\rightarrow M$ such that $\pi\iota=\Id_M$, then we say that $M$ is isomorphic to a direct summand of $X$ and we write $M\mid X$.
\end{notation}

When $\cC=kG-$mod for a finite group $G$, $\cD=kH-$mod for $H\leq G$, and $\F$ and $\E$ are induction and restriction respectively, there are several equivalent ways to detect when $M\mid\F\E(M)$ which go by the name of Higman's criterion. Brou\'{e} recognized that Higman's criterion is simply a statement about exact, biadjoint functors valid in a much more general setting (the following theorem allows $\cC$ and $\cD$ to be any $R$-linear abelian or triangulated categories where $R$ is a commutative ring with $1$). The trace map $\Tr_{\E}^{\F}(M):\End(\E(M))\rightarrow\End(M)$ is defined as \cite[Definition 6.6]{Broue}: $$\Tr_{\E}^{\F}(M)(\beta)=\varepsilon_M\F(\beta)\eta_M.$$
In particular, $\zeta_M=\Tr_{\E}^{\F}(M)(\Id_{\E(M)})$. 
\begin{theorem}\cite[Theorem 6.8]{Broue}\label{Higmans}
For an object $M\in\cC$, the following are equivalent.
\begin{enumerate}
\item $M\mid\F\E(M)$;
\item $M\mid\F(N)$ for some $N\in\cD$;
\item The morphism $\Id_M$ is in the image of $\Tr_{\E}^{\F}(M)$;
\item The morphism $\eta_M:M\rightarrow\F\E(M)$ has a left inverse;
\item The morphism $\varepsilon_M:\F\E(M)\rightarrow M$ has a right inverse.
\end{enumerate}
\end{theorem}\noindent
There are two more conditions in Brou\'{e}'s theorem generalizing the notion of relative projectivity and injectivity of maps, but we omit these here. Note that the criteria in Theorem \ref{Higmans} do not imply that $\zeta_M$ has an inverse.

\begin{corollary}\label{captainobvious}
Let $M\in\cC$. If $\zeta_M$ is an isomorphism then $M\mid\F\E (M)$.
\end{corollary}

\begin{lemma}\label{mofo}
Let $\cB$ be a block of $\cC$. Suppose there exists a simple module $L\in\cB$ such that $\zeta_L\neq 0$. Then $M\mid\F\E (M)$ for every $M\in\cB$.
\end{lemma}
\begin{proof} It suffices to consider the case that $M$ is indecomposable. Consider diagram (\ref{local}) above with $L$ taken to be any simple module in the head of $M$, then make the same argument as in the proof of Lemma \ref{simple} to conclude that $\zeta_M$ is an isomorphism. By Corollary \ref{captainobvious}, then $M\mid\F\E (M)$.
\end{proof}

Given what conditions on $M$ does $M\mid\F\E(M)$ imply that $\zeta_M$ is an isomorphism? 
A condition is given in the proof of \cite[Corollary 3.3]{GriffethJuteau} which is concerned with certain $1$-dimensional modules over Hecke algebras but is more generally valid. Here is the statement and an alternative proof in our more general set-up.

\begin{lemma}\label{zeta} Suppose $\dim\End(\E (M))=1$. Then $M\mid \F\E(M)$ if and only if $\zeta_M$ is a nonzero multiple of $\Id_M$. 
\end{lemma}
\begin{proof}
 Since $\dim\End(\E(M))=1$, for any $\beta\in\End(\E(M))$ we have $\beta=b\cdot\Id_{\E(M)}$ for some $b\in k$. Then:
$$
\Tr^\F_\E(M)(\beta)=\varepsilon_M(\F(b\cdot \Id_{\E(M)})\eta_M=\varepsilon_M(b\cdot\Id_{\F\E(M)})\eta_M=b\varepsilon_M\eta_M=b\zeta_M
$$
Therefore $\Id_M$ is in the image of $\Tr^\F_\E(M)$ if and only if $\zeta_M$ is a nonzero multiple of $\Id_M$. By Theorem \ref{Higmans}, $M\mid\F\E(M)$ if and only if $\Id_M$ is in the image of $\Tr^\F_\E(M)$.
\end{proof}
\noindent The image of the trace map is a two-sided ideal in $\End(M)$ \cite[Proposition 6.7]{Broue}, so in the event the conditions in Lemma \ref{zeta} all hold then $\dim\End(M)=1$ as well.

\section{Vertices for Cherednik and Hecke algebras of symmetric groups}
 The ground field for the rest of the paper is $\mathbb{C}$. 
\subsection{Vertices for category $\cO$ of the Cherednik algebra} 
The material in this section is mostly a copy-paste of the definition and basic properties of vertices from categories such as $kG-$mod for $G$ a finite group together with group induction and restriction, or unipotent representations of a finite group of Lie type in cross characteristic together with Harish-Chandra induction and restriction.  We include detailed proofs for completeness. 

 Let $W$ be a complex reflection group, let $c:\{\hbox{Reflections in W}\}\rightarrow \mathbb{C}$ be a conjugation-invariant function, and let $\cO_c(W)$ be the category $\cO$ of the rational Cherednik algebra defined in \cite{GGOR}.
 This is a highest weight category \cite{GGOR}, so it occurs as the category of finitely generated modules for a quasi-hereditary algebra \cite{CPS}; it has simple, Verma, and projective indecomposable modules in bijection with $\Irr_\C(W)$ \cite{GGOR}.
 
 Let $\underline{W}\subseteq W$ be a parabolic subgroup. Parabolic induction and restriction functors 
$$
{^\cO\Ind}_{\underline{W}}^W:\cO_c(\underline{W})\longrightarrow\cO_c(W)\quad\hbox{and}\quad{^\cO\Res}_{\underline{W}}^W:\cO_c(W)\longrightarrow\cO_c(\underline{W})
$$
\noindent  were defined by Bezrukavnikov and Etingof \cite{BezrukavnikovEtingof}.  
The functors ${^\cO\Ind}_{\underline{W}}^W$ and ${^\cO\Res}_{\underline{W}}^W$ are exact and biadjoint \cite{BezrukavnikovEtingof},\cite{Shan},\cite{Losev2013}. Therefore:
\begin{lemma}For any parabolic subgroup $\underline{W}\subseteq W$, Theorem \ref{Higmans} applies to\\ $\cC=\cO_c(W)$ and $\cD=\cO_c(\underline{W})$ with $\E= {^\cO\Res}_{\underline{W}}^W$ and $\F={^\cO\Ind}_{\underline{W}}^W$,
giving equivalent conditions for when $M\mid{^\cO\Ind}_{\underline{W}}^W\;{^\cO\Res}_{\underline{W}}^W M$.
\end{lemma}
\begin{definition} A \textit{vertex} of $M\in\cO_c(W)$ is a minimal parabolic subgroup $\underline{W}\subseteq W$ such that $M\mid{^\cO\Ind}_{\underline{W}}^W N$ for some $N\in\cO_c(\underline{W})$.
\end{definition}

In the classical setting of $kG$-mod where $G$ is a finite group and $k$ has characteristic $p$, it is the Mackey formula that implies the uniqueness of the vertices of indecomposable $kG$-modules up to conjugacy.
Recall that if $H$ and $K$ are subgroups of a finite group $G$ and $V$ is a $kH$-module, then the Mackey formula states: 
$$\Res^G_K\Ind^G_{H}(V)=\bigoplus_{u\in K\backslash G/H}\Ind^K_{uHu^{-1}\cap K}\Res^{uHu^{-1}}_{uHu^{-1}\cap K}({^uV})$$
where ${^uV}:=u\otimes V$, a natural $uHu^{-1}$-module, see e.g. \cite[Lemma 8.7]{Alperin}. 
In the Hecke and Cherednik algebra versions of the Mackey formula, one takes $W$ in place of $G$ and two parabolic subgroups $W_1$ and $W_2$ in place of $H$ and $K$; the group induction and restriction functors are replaced by the appropriate parabolic induction and restriction functors. Kuwabara-Miyachi-Wada prove the Mackey formula for Hecke and Cherednik algebras when $W=G(\ell,1,n)$ (for $H_q(W)-$mod see \cite[Theorem 3.12]{KMW} and for $\cO_c(W)$ see \cite[Theorem 5.6]{KMW}), and they conjecture that the Mackey formula holds in $\cO_c(W)$ for arbitrary complex reflection groups $W$ \cite[Conjecture 0.1]{KMW}. Losev and Shelley-Abrahamson prove that when $W$ is a finite Coxeter group, the Mackey formula holds for  
$\cO_c(W)$ \cite[Proposition 2.7.2]{LosevSA} by lifting it using the $\KZ$ functor from the formula for the Hecke algebra known in this case by \cite[Proposition 9.1.8]{GeckPfeiffer}.
The precise formulas read \cite{GeckPfeiffer},\cite{KMW},\cite{LosevSA}:

\begin{align*}
\Res^{H_q(W)}_{H_q(W_2)}\Ind^{H_q(W)}_{H_q(W_1)}&\cong \bigoplus_{u\in W_2\backslash W/W_1}\Ind^{H_q(W_2)}_{H_q(W_2\cap uW_1u^{-1})}\circ \; u(-)\circ\Res^{H_q(W_1)}_{H_q(u^{-1}W_2u\cap W_1)}\\
{^\cO\Res}^W_{W_2}{^\cO\Ind}^W_{W_1}&\cong \bigoplus_{u\in W_2\backslash W/W_1}{^\cO\Ind}^{W_2}_{W_2\cap uW_1u^{-1}}\circ u(-)\circ{^\cO\Res}^{W_1}_{u^{-1}W_2u\cap W_1}
\end{align*}
The functor $u(-)$ is an equivalence induced by conjugation by $u$. \textit{From now on, we will always assume the Mackey formula holds for $\cO_c(W)$ and $H_q(W)-$mod.} In particular, it holds for $W=S_n$ since $S_n$ is a Coxeter group and $S_n=G(1,1,n)$.

Now as in \cite[Theorem 5.1.2]{Linckelmann} the Mackey formula implies uniqueness of vertices up to conjugacy; the proof for $kG$-modules also works for Hecke and Cherednik algebras. We give the proof anyway:
\begin{lemma}\label{Mackey}
 Let $M\in\cO_c(W)$ or $H_q(W)-$mod. Then a vertex of $M$ is unique up to $W$-conjugacy.
\end{lemma}
\begin{proof} Let $\cC(W)$ be $\cO_c(W)$ or $H_q(W)-$mod and let $M\in\cC(W)$. 
Write $\Ind$ and $\Res$ for the appropriate parabolic induction and restriction functors for the chosen category. Let $W'$ be a vertex of $M$. By Theorem \ref{Higmans}, $M\mid\Ind^W_{W'}\Res_{W'}^W M$. Let $D\in\cC(W')$ be a direct summand of $\Res^W_{W'}M$ such that $M\mid\Ind^W_{W'}D$.  Suppose $W''$ is another vertex of $M$ and let $E\in\cC(W'')$ such that $M\mid\Ind_{W''}^W E$. Then
$$ D\mid \Res_{W'}^W\Ind_{W''}^W E=\bigoplus_{u\in W'\backslash W/W''}\Ind^{W'}_{W'\cap uW''u^{-1}}\circ\; u(-)\circ\Res^{W''}_{u^{-1} W'u\cap W''}E$$
The minimality of $W'$ implies that $D$ is not a direct summand of $\Ind_{W'\cap uW''u^{-1}}^{W'}X$ whenever $\left(W'\cap uW''u^{-1}\right)\subsetneq W'$, since otherwise $M\mid\Ind^W_{W'\cap uW''u^{-1}}X$ by transitivity. This forces $W'\leq uW''u^{-1}$ for some $u$. Repeating the argument with the roles of $W'$ and $W''$ switched, we conclude that $W'$ and $W''$ are conjugate.
\end{proof}

The vertices of projective indecomposable modules are closely related to the branching rules for simple modules. 
\begin{definition}\cite{BezrukavnikovEtingof} A module $M\in\cO_c(W)$ is called \textit{cuspidal} if ${^\cO\Res}_{\underline{W}}^W M=0$ for all parabolics $\underline{W}\subsetneq W$.
\end{definition}

\begin{definition}\cite{LosevSA} Let $L\in\cO_c(W)$ be a simple module. A \textit{cuspidal support} of $L$ is a pair $(W',L')$, where $W'\subseteq W$ is a parabolic subgroup and $L'\in\cO_c(W')$ is a simple cuspidal module, such that ${^\cO\Ind}_{W'}^W L'\onto L$.
\end{definition}\noindent
The Mackey formula implies that cuspidal supports of simple modules are unique up to $W$-conjugacy \cite[Proposition 3.1.2]{LosevSA}. 

 The following lemma is well-known for unipotent representations of a finite reductive group in cross-characteristic endowed with Harish-Chandra induction and restriction, see e.g. \cite[Proposition 10.6]{BeijingLectures}, and the proof for Cherednik algebras works exactly the same way. Part of the statement was shown in \cite[Lemma 3.2]{GriffethJuteau}.
\begin{lemma}\label{PIM}
Let $L\in\cO_c(W)$ be a simple module and $P$ its projective cover, and let $(W',L')$ be a cuspidal support of $L$. Let $P'$ be the projective cover of $L'$. Then $P\mid{^\cO\Ind}_{W'}^W P'$ and $W'$ is a vertex of $P$.
\end{lemma}

\begin{proof}
Since ${^\cO\Ind}_{W'}^W$ and ${^\cO\Res}_{W'}^W$ are exact and biadjoint, they take projectives to projectives. Since $P'\onto L'$ and ${^\cO\Ind}_{W'}^W$ is exact, ${^\cO\Ind}_{W'}^W P'\onto{^\cO\Ind}_{W'}^W L'\onto L$ is a surjection onto $L$. The universal property of projectives then yields ${^\cO\Ind}_{W'}^W P'\onto P$, and since $P$ is projective, this implies $P\mid {^\cO\Ind}_{W'}^W P'$. Now, suppose $\underline{W}\subseteq W'$ and $M\in\cO_c(\underline{W})$ such that $P\mid{^\cO\Ind}_{\underline{W}}^W M$. Then ${^\cO\Ind}_{\underline{W}}^W M\onto L$, so by adjointness $0\neq \Hom({^\cO\Ind}_{\underline{W}}^W M, L)\cong\Hom(M,{^\cO\Res}_{\underline{W}}^W L)$, implying that $W'=\underline{W}$.
\end{proof}
\noindent As in \cite[Proposition 10.6]{BeijingLectures} we then recover the statement that all cuspidal supports $(W',L')$ of a simple module $L\in\cO_c(W)$ are $W$-conjugate. If $\mathrm{rank}(W')=\mathrm{rank}(W)-j$ then we will refer to $j$ as the \textit{cuspidal depth} of $L$. 
Since vertices and cuspidal supports are unique up to conjugacy, we will speak from now on of \textit{the} vertex of a module $M$ and \textit{the} cuspidal support of a simple module $L$. 

\subsection{The KZ functor}
For any complex reflection group $W$ there is a functor $$\KZ:\cO_c(W)\rightarrow H_q(W)-\mathrm{mod}$$ (where $H_q(W)-$mod denotes the category of finite-dimensional $H_q(W)$-modules) which is exact and represented by the object $P_{\KZ}={^\cO\Ind}_{1}^W\C$ \cite{GGOR}. This functor has very strong properties: $\KZ$ is fully faithful on projectives \cite{GGOR}, and $\KZ$ is essentially surjective \cite{LosevKZ}. The Double Centralizer Theorem \cite[Theorem 5.16]{GGOR} shows that blocks of $\cO_c(W)$ are in bijection with blocks of $H_q(W)-$mod \cite[Corollary 5.18]{GGOR}. 

Shan showed that for any parabolic $\underline{W}\subseteq W$ there are functor isomorphisms \cite{Shan}:
$$\KZ\;{^\cO\Ind}_{\underline{W}}^W\cong\Ind_{H_q(\underline{W})}^{H_q(W)}\;\underline{\KZ}\quad\hbox{and}\quad \underline{\KZ}\;{^\cO\Res}_{\underline{W}}^W\cong\Res_{H_q(\underline{W})}^{H_q(W)}\;\KZ$$
where $\underline{\KZ}$ denotes the $\KZ$ functor $\cO_c(\underline{W})\rightarrow H_q(\underline{W})-$mod. Since $\KZ$ respects direct sums, this has an immediate consequence for vertices:
\begin{lemma}\label{deeper}
If $M\mid{^\cO\Ind}_{\underline{W}}^W{^\cO\Res}_{\underline{W}}^W M$ then $\KZ(M)\mid \Ind_{H_q(\underline{W})}^{H_q(W)}\Res_{H_q(\underline{W})}^{H_q(W)}\KZ(M)$ for any $M\in\cO_c(W)$.
\end{lemma}

\begin{lemma}\label{GJ}\cite[Lemma 3.2]{GriffethJuteau} Let $P\in\cO_c(W)$ be a projective indecomposable module. The vertex of $P$ is equal to the vertex of $\KZ(P)$.
\end{lemma}
\begin{proof} As observed in \cite[Lemma 3.2]{GriffethJuteau}, it is basically immediate that  
 $$P\mid{^\cO\Ind}_{\underline{W}}^W{^\cO\Res}_{\underline{W}}^W P\iff \KZ(P)\mid \Ind_{H_q(\underline{W})}^{H_q(W)}\Res_{H_q(\underline{W})}^{H_q(W)}\KZ(P)$$ but we give full details here. The direction ``$\implies$" is Lemma \ref{deeper}.
For ``$\impliedby$:" suppose that $\KZ(P)\mid \Ind_{H_q(\underline{W})}^{H_q(W)}\Res_{H_q(\underline{W})}^{H_q(W)}\KZ(P)$. Then there are maps
$$\KZ(P)\stackrel{\iota}{\longrightarrow}\Ind_{H_q(\underline{W})}^{H_q(W)}\Res_{H_q(\underline{W})}^{H_q(W)}\KZ(P)\stackrel{\pi}{\longrightarrow}\KZ(P)$$
such that $\pi\iota=\Id_{\KZ(P)}$. We have $\Ind_{H_q(\underline{W})}^{H_q(W)}\Res_{H_q(\underline{W})}^{H_q(W)}\KZ(P)=\KZ\left({^\cO\Ind}_{\underline{W}}^W{^\cO\Res}_{\underline{W}}^W P\right)$ by \cite{Shan}. Moreover, ${^\cO\Ind}_{\underline{W}}^W{^\cO\Res}_{\underline{W}}^W P$ is projective since parabolic restriction and induction take projectives to projectives \cite{Shan}. Since $\End(P)\cong \End(\KZ(P))$ \cite{GGOR}, the maps $\iota$ and $\pi$ lift to maps 
$$P\stackrel{\widetilde{\iota}}{\longrightarrow} {^\cO\Ind}_{\underline{W}}^W{^\cO\Res}_{\underline{W}}^W P
\stackrel{\widetilde{\pi}}{\longrightarrow} P$$
such that $\KZ(\widetilde{\pi})=\pi$ and $\KZ(\widetilde{\iota})=\iota$.
The composition $\tilde{\pi}\tilde{\iota}=\Id_P$ because  $\KZ(\tilde{\pi}\tilde{\iota})=\pi\iota=\Id_{\KZ(P)}$ and $\KZ$ is injective on $\End(P)$. This shows $P\mid{^\cO\Ind}_{\underline{W}}^W{^\cO\Res}_{\underline{W}}^W P$.
\end{proof}

\subsection{Blocks and cuspidal supports for $\cO_c(S_n)$}
We recall some facts about $\cO_c(S_n)$. 
Fix $e\in\mathbb{N}_{\geq 2}$, set $c=\frac{r}{e}>0$ with $\mathrm{gcd}(r,e)=1$, and set $q=\exp(2\pi i c)$.

We use the convention that $(n)$ is the trivial representation of $S_n$. The category $\cO_c(S_n)$ has a unique simple module $L_\lambda$, Verma module $\Delta_\lambda$, and projective indecomposable module $P_\lambda$ for each partition $\lambda$ of $n$. The $\KZ$ functor sends $\Delta_\lambda$ to the Specht module labeled by $\lambda$, and sends $L_\lambda$ to the simple module $D_\lambda$ if $\lambda$ is $e$-restricted and otherwise to $0$ \cite{GGOR}. (Recall that an $e$-restricted partition is a partition $\lambda=(\lambda_1,\lambda_2,\dots)$ satisfying $\lambda_i-\lambda_{i+1}<e$, and such partitions parametrize the simple $H_q(S_n)$-modules). 
The blocks of $H_q(S_n)$, and therefore $\cO_c(S_n)$, are parametrized by $e$-cores: the partitions $\lambda$ labeling simple, standard, and projective indecomposable modules
 in the block $\cB_{\rho,w}$ of $\cO_c(S_n)$ are exactly the partitions of size $n=|\rho|+ew$ with $e$-core $\rho$ and $e$-weight $w$, the latter being defined as the number of $e$-hooks removed successively from the rim of $\lambda$ to obtain $\rho$ (see e.g. \cite{JamesKerber}) \cite[Theorem 4.13]{DipperJames}. If $\sigma=(\sigma_1,\sigma_2,\dots)$ is a partition of $w$ we write $e\sigma$ for the partition $(e\sigma_1,e\sigma_2,\dots)$, and given partitions $\mu=(\mu_1,\mu_2,\dots)$ and $\nu=(\nu_1,\nu_2,\dots)$ we write $\mu+\nu$ for the partition $(\mu_1+\nu_1,\mu_2+\nu_2,\dots)$.
 
The category $\cO_c(S_n)$ has a cuspidal simple module $L$ if and only if $n=e$, in which case $L=L_{(e)}$  \cite{BerestEtingofGinzburg}. The category $\cO_c\left(S_e^{\times k}\right)=\cO_c(S_e)^{\otimes k}$ then has a unique cuspidal simple module $L_{(e)}^{\otimes k}$. All parabolic subgroups of $S_n$ are of the form $S_{m_1}\times S_{m_2}\times\dots\times S_{m_s}$ with $\sum_{j=1}^sm_j=n$. Since we work up to conjugacy, when $m_j=1$ we will omit $S_1=\{1\}$ from the notation. Thus the parabolics $S_e^{\times k}$ are the only parabolic subgroups of $S_n$ whose category $\cO_c$ affords a cuspidal. We will abuse terminology and refer to $S_e^{\times k}$ as the cuspidal support when we mean $(S_e^{\times k},L_{(e)}^{\otimes k})$. Let $\lambda$ be a partition of $n$ and write $\lambda=e\sigma+\nu$ where $\nu$ is $e$-restricted and $\sigma$ is a partition of some $k\geq 0$. Wilcox showed that the cuspidal support of $L_\lambda$ is $S_e^{\times k}$ \cite[Theorem 1.6]{Wilcox}. The simples $L_\lambda\in\cB_{\rho,w}$ of minimal cuspidal depth in the block $\cB_{\rho,w}$ are labeled by partitions of the form $\lambda=e\sigma+\rho$ where $\sigma$ is a partition of $w$. For such a simple, we have:
\begin{equation}\label{minsupp}
{^\cO\Res^{S_n}_{S_e^{\times w}}}L_\lambda=\left(L_{(e)}^{\otimes w}\right)^{\oplus a_{\lambda}}\quad\hbox{and}\quad{^\cO\Res^{S_n}_{\underline{W}}}L_\lambda=0\hbox{ for any }\underline{W}\subsetneq S_e^{\times w}.
\end{equation}
\noindent where $a_\lambda>0$ is some multiplicity. Wilcox identified the subquotient category spanned by the simples in $\cO_c(S_n)$ of a fixed cuspidal depth:
\begin{theorem}\label{Wilcox's equivalence}\cite[Theorem 1.8]{Wilcox} The Serre subquotient category of $\cO_c(S_n)$ consisting of modules with cuspidal support $S_e^{\times w}$ is equivalent to the category of finite-dimensional modules over $\mathbb{C}[S_w]\otimes H_q(S_{n-ew})$ with $q=\exp(2\pi ic)$. If  $\tau_\sigma$ is a simple representation of $\mathbb{C}[S_w]$ and $D_\nu$ is a simple representation of $H_q(S_{n-ew})$ then the simple representation in $\cO_c(S_n)$ corresponding to $\tau_\sigma\otimes D_\nu$ under this equivalence is $L_{e\sigma+\nu}$.
\end{theorem}
 
\subsection{The Dipper-Du Conjecture} 
We now establish the analogous statement to Dipper-Du's conjecture for $\cO_c(S_n)$, then re-establish Dipper-Du's conjecture for $H_q(S_n)$ over $\C$.
\begin{theorem}\label{VerticesCherednik}
Let $\cB=\cB_{\rho,w}$ be a block of $\cO_c(S_n)$ of $e$-weight $w$ and $e$-core $\rho$. The vertices of all modules in $\cB$ are contained in $S_e^{\times w}$, and the simple and projective modules $L_\lambda$, $P_\lambda\in\cB$ such that $\lambda=e\sigma+\rho$, $\sigma$ a partition of $w$, have $S_e^{\times w}$ as their vertex. Moreover, $\{S_e^{\times k}\mid 0\leq k\leq\lfloor\frac{n}{e}\rfloor\}$ comprises the vertices of $\cO_c(S_n)$.
\end{theorem}
\begin{proof}
The simple modules $L_\lambda$ in $\cB$ of minimal cuspidal depth are those such that $\lambda=e\sigma+\rho$ for $\sigma$ a partition of $w$. 
Since $\rho$ is an $e$-core, $D_\rho\in\mathcal{H}_q(S_{|\rho|})-$mod is projective and in a block of $\mathcal{H}_q(S_{|\rho|})-$mod by itself. The block of $\mathbb{C}[S_w]\otimes \mathcal{H}_q(S_{|\rho|})-$mod corresponding under the equivalence of Theorem \ref{Wilcox's equivalence} to the Serre subcategory spanned by the simple modules in $\cB$ of minimal cuspidal depth is therefore equivalent to $\mathbb{C}[S_w]-$mod. If $\underline{W}\subseteq S_n$ is a parabolic subgroup and $L\in\cO_c(\underline{W})$ is a simple module, then the cuspidal depth of a simple constituent of ${^\cO\Ind}^{S_n}_{\underline{W}} L$ can never be larger than the cuspidal depth of the head of ${^\cO\Ind}^{S_n}_{\underline{W}} L$. It follows that if $\lambda=e\sigma+\rho$ with $\rho$ an $e$-core and $\sigma$ a partition of $w$, then: $$L_\lambda\mid{^\cO\Ind^{S_n}_{S_e^{\times w}}}L_{(e)}^{\otimes w}.$$
Combined with equation (\ref{minsupp}) above, this shows that $S_e^{\times w}$ is the vertex of $L_\lambda$ for every $L_\lambda$ in $\cB_{\rho,w}$ of minimal cuspidal depth.



To finish the proof of the theorem it is enough to show that there is a simple module $L_\lambda\in\cB$ such that $\zeta_{L_\lambda}\neq 0$ and then apply Lemma \ref{mofo}. To this end, we now consider $L_{e(w)+\rho}$. First, let us explain what happens when $\rho=\emptyset$, so that $\lambda=e(w)=(ew)$ is the trivial representation of $S_{ew}$. By \cite{ShanVasserot}, 
$${^\cO\Res^{S_{ew}}_{S_e^{\times w}}}L_{(ew)}=L_{(e)}^{\otimes w}.$$
Applying Lemma \ref{zeta} gives that $\zeta_{L_{(ew)}}$ is an isomorphism. 
Lemma \ref{mofo} then implies that $M\mid{^\cO\Ind^{S_{ew}}_{S_e^{\times w}}}{^\cO\Res^{S_{ew}}_{S_e^{\times w}}} M$ for all $M\in\cB_{\emptyset,w}$. Thus if $\rho=\emptyset$, we are done.

From now on, assume $\rho\neq \emptyset$. We will copy the strategy of \cite{Whitley} by considering a relevant block of the category $\cO_c(S_{|\rho|}\times S_{ew})$ as an intermediate step.
Consider the block $\cB_{\rho,0}\otimes\cB_{\emptyset,w} \subset\cO_c(S_{|\rho|})\otimes \cO_c(S_{ew})=\cO_c(S_{|\rho|}\times S_{ew})$. By Lemma \ref{PIM} the vertex of $L_\rho=\Delta_\rho=P_\rho$ is $\{1\}$. Thus the vertex of $L_\rho\otimes M$ for any $M\in\cB_{\emptyset,w}$ is just the vertex of $M$ (since we ignore copies of $\{1\}$ in a parabolic).

Next, we pre- and post-compose the induction and restriction functors with the functors of inclusion and projection from and to the desired blocks. Define functors $\E$ and $\F$ by:
\begin{align*}
\E&=\mathrm{Pr}_{\cB_{\rho,0}\otimes\cB_{\emptyset,w}}\;{^\cO\Res}_{S_{|\rho|}\times S_{ew}}^{S_n}\;\mathrm{Incl}_{\cB_{\rho,w}}\\
\F&=\mathrm{Pr}_{\cB_{\rho,w}}\;\;{^\cO\Ind}_{S_{|\rho|}\times S_{ew}}^{S_n}\;\mathrm{Incl}_{\cB_{\rho,0}\otimes\cB_{\emptyset,w}}
\end{align*}
Here, $\mathrm{Pr}_{\cB_{\rho,w}}$ is projection from $\cO_c(S_n)$ onto the block $\cB_{\rho,w}$ and $\mathrm{Incl}_{\cB_{\rho,w}}$ is inclusion of the block $\cB_{\rho,w}$ into $\cO_c(S_n)$, a biadjoint pair of functors; and similarly with the functors $\mathrm{Pr}_{\cB_{\rho,0}\otimes\cB_{\emptyset,w}}$ and $\mathrm{Incl}_{\cB_{\rho,0}\otimes\cB_{\emptyset,w}}$ for the block $\cB_{\rho,0}\otimes\cB_{\emptyset,w}$ of $\cO_c(S_{|\rho|}\times S_{ew})$. By \cite[Theorem IV.8.1]{MacLane}, $\E$ and $\F$ are biadjoint. Moreover $\E$ and $\F$ are exact as each functor in the compositions defining them is exact. Let $\zeta=\varepsilon\eta$ be the natural transformation of the identity functor on $\cB_{\rho,w}$ arising from the biadjunction between $\E$ and $\F$.

We claim that $\E(L_{e(w)+\rho})=L_\rho\otimes L_{(ew)}$. We know that the module $\F \left(L_\rho\otimes L_{(ew)}\right)$ is semisimple by semisimplicity of the subcategory of $\cB_{\rho,w}$ generated by the simples of minimal cuspidal depth in the block. 
In the Grothendieck group we can write $[L_\lambda]=[\Delta_\lambda]+\sum\limits_{\mu\vartriangleleft\lambda}c_\mu[\Delta_\mu]$ for some $c_\mu\in\Z$ \cite{GGOR}. 
 The induction rule for $[^\cO\Ind_{\underline{W}}^{S_n}\Delta_\chi]$ is just the group induction rule for $\Ind_{\underline{W}}^{S_n}\chi$ \cite{BezrukavnikovEtingof}.
For any partitions $\lambda,\mu,\nu$, the Littlewood-Richardson coefficient $c_{\lambda,\mu}^{\nu}\neq 0$ implies $\nu_1\leq \lambda_1+\mu_1$. It follows that $[\Delta_{e(w)+\rho}]$ does not occur in $[\F(\Delta_{\rho}\otimes \Delta_{\tau})]$ for any $\tau\neq (ew)$.  Since $e(w)+\rho$ is the maximal partition in the block in dominance order, then $[L_{e(w)+\rho}]$ does not occur in $[\F \left(L_\rho\otimes L_{\tau}\right)]$ for any other $\tau\neq (ew)$. Also, we have $c_{\rho,(ew)}^{\rho+e(w)}=1$, thus $L_{e(w)+\rho}\mid\F(L_\rho\otimes L_{(ew)})$ with multiplicity $1$.
 So we have \begin{align*}1&=\dim\Hom(L_{e(w)+\rho},\F(L_\rho\otimes L_{(ew)}))=\dim\Hom(\E(L_{e(w)+\rho}), L_\rho\otimes L_{(ew)}),\\0&=\dim\Hom(L_{e(w)+\rho},\F(L_\rho\otimes L_{\tau}))=\dim\Hom(\E(L_{e(w)+\rho}),L_\rho\otimes L_{\tau})\end{align*} for $\tau\neq(ew)$ and so $\E(L_{e(w)+\rho})$ is indecomposable with simple head $L_\rho\otimes L_{(ew)}$. But its composition factors must have the same cuspidal support as $L_{e(w)+\rho}$ (they cannot have bigger depth and there is no smaller), therefore by previous remarks $\E(L_{e(w)+\rho})$ is semisimple. 
Therefore $\E(L_{e(w)+\rho})=L_\rho\otimes L_{(ew)}$ and we may apply Lemma \ref{zeta} obtaining that $\zeta_{L_{e(w)+\rho}}$ is an isomorphism; Lemma \ref{mofo} then implies that $M\mid\F\E\left( M\right)$ for all $M\in\cB_{\rho,w}$.

Therefore $M\mid{^\cO\Ind^{S_n}_{S_e^{\times w}}}{^\cO\Res^{S_n}_{S_e^{\times w}}}M$ for all $M\in\cB_{\rho,w}$, and by (\ref{minsupp}) above, $S_e^{\times w}$ is the minimal parabolic for which such a statement holds. If $a<e$ then every $L_\lambda$ in $\cO_c(S_a)$ is projective and in a block by itself; for any $M\in\cO_c(S_e)^{\otimes w}$ then, the vertex of $M$ is $S_e^{\times k}$ for some $k\leq w$. Thus the set of vertices of $\cO_c(S_n)$ is contained in the set of parabolics $\{S_e^{\times k}\mid 0\leq k\leq \lfloor \frac{n}{e}\rfloor\}$. 
By \cite[Theorem 1.6]{Wilcox}, for every $0\leq k\leq \lfloor \frac{n}{e}\rfloor$ there exists a partition $\lambda$ of $n$ such that $L_\lambda$ has cuspidal support $S_e^{\times k}$. Then $S_e^{\times k}$ is the vertex of its projective cover $P_\lambda$ by Lemma \ref{PIM}. This shows the set of vertices of $\cO_c(S_n)$ contains the set of parabolics $\{S_e^{\times k}\mid 0\leq k\leq \lfloor \frac{n}{e}\rfloor\}$. We are done.
\end{proof}

\begin{theorem}[\textit{Dipper-Du conjecture over $\C$}]\label{VerticesHecke} Let $\mathsf{B}$ be a weight $w$ block of $H_q(S_n)-\mathrm{mod}$.
The vertices of all modules in ${\mathsf{B}}$ are contained in $S_e^{\times w}$, and the modules $\KZ(P_\lambda)\in{\mathsf{B}}$ such that $\lambda=e\sigma+\rho$, $\sigma$ a partition of $w$, have $S_e^{\times w}$ as their vertex. Moreover, $\{S_e^{\times k}\mid 0\leq k\leq\lfloor\frac{n}{e}\rfloor\}$ comprises the vertices of $H_q(S_n)$.
\end{theorem}
\begin{proof}
Let $\cB$ be the block of $\cO_c(S_n)$ such that $\KZ(\cB)=\mathsf{B}$. 
Let $P_\lambda$ be the projective cover of $L_\lambda\in{\cB}$ where $L_\lambda\mid\Ind^{S_n}_{S_e^{\times w}}L_{(e)}^{\otimes w}$. Then $S_e^{\times w}$ is the vertex of $P_\lambda$ since $(S_e^{\times w}, L_e^{\otimes w})$ is the cuspidal support of $L_\lambda$ \cite{Wilcox}. Then by Lemma \ref{GJ} $S_e^{\times w}$ is the vertex of $\KZ(P_\lambda)$.
Moreover, for any $N\in\mathsf{B}$ there exists $M\in\cB$ such that $N\cong \KZ(M)$ by essential surjectivity of $\KZ$ \cite{LosevKZ}. By Lemma \ref{deeper} the vertex of $\KZ(M)$ is contained in the vertex of $M$. It then follows from Theorem \ref{VerticesCherednik} that the vertex of $\KZ(M)$ is a subgroup of $S_e^{\times k}$, where $S_e^{\times k}$ is the vertex of $M$. Since $H_q(S_m)-$mod is semisimple for $1\leq m<e$, this shows that the set of vertices of $H_q(S_n)-$mod is contained in the set $\{S_e^{\times k}\mid 0\leq k\leq\lfloor\frac{n}{e}\rfloor\}$. But (as just used in the proof of Theorem \ref{VerticesCherednik}) the set of vertices of projective indecomposable modules $P_\mu$ of $\cO_c(S_n)$ is equal to $\{S_e^{\times k}\mid 0\leq k\leq\lfloor\frac{n}{e}\rfloor\}$, and by Lemma \ref{GJ} the vertex of $P_\mu$ is the same as the vertex of $\KZ(P_\mu)$. Therefore the set of vertices of $H_q(S_n)-$mod is equal to $\{S_e^{\times k}\mid 0\leq k\leq\lfloor\frac{n}{e}\rfloor\}$.
\end{proof}

\subsection{The vertices of simple modules in $\cO_c(S_n)$}

The category $\cO_c(W)$ has enough projectives and has finite global dimension \cite{GGOR}, so any module $M$ in $\cO_c(W)$ has a finite projective resolution $P_\bullet$ which is unique up to direct summands of trivial complexes $0\rightarrow Q\stackrel{\sim}{\rightarrow}Q\rightarrow 0$. If $P_\bullet$ does not contain any such trivial summands then $P_\bullet$ is said to be a minimal projective resolution. By replacing $M$ by its minimal projective resolution, we can get a lower bound on the vertex of $M$. 

\begin{lemma}\label{lb proj res} Let $P_\bullet=\dots \rightarrow P_n\rightarrow P_{n-1}\rightarrow\dots \rightarrow P_0\rightarrow 0$ be a minimal projective resolution of a module $M\in\cO_c(W)$. Then $M\mid {^\cO\Ind}_{\underline{W}}^W\;{^\cO\Res}_{\underline{W}}^W M$ if and only if \\$P_\bullet\mid{^\cO\Ind}_{\underline{W}}^W\;{^\cO\Res}_{\underline{W}}^W P_\bullet$ as complexes. In particular, if $M\mid {^\cO\Ind}_{\underline{W}}^W\;{^\cO\Res}_{\underline{W}}^W M$ then $Q\mid{^\cO\Ind}_{\underline{W}}^W\;{^\cO\Res}_{\underline{W}}^W Q$ for every projective indecomposable module $Q$ in $P_\bullet$.
\end{lemma}
\begin{proof} This follows from Theorem \ref{Higmans} applied to $D_b(\cO_c(W))$ and $D_b(\cO_c(\underline{W}))$.
\end{proof}
Now let $W=S_n$, $2\leq e\leq n$  and $c=\frac{r}{e}>0$, $\mathrm{gcd}(r,e)=1$.

\begin{lemma}\label{wt 1}
Let $L_\lambda$ be any simple module in the principal block $\cB_{\emptyset,1}$ of $\cO_c(S_e)$. Then the vertex of $L_\lambda$ is $S_e$.
\end{lemma}
\begin{proof}
The structure of the block $\cB_{\emptyset,1}$ is completely known, see \cite{BerestEtingofGinzburg},\cite{Rouquier}. It is easy to calculate the minimal projective resolution of any simple $L_\lambda\in\cB_{\emptyset,1}$; the final nonzero term of this resolution is $P_{(e)}$. The simple $L_{(e)}$ is cuspidal by \cite{BerestEtingofGinzburg}, so by Lemma \ref{PIM} the vertex of $P_{(e)}$ is $S_e$. Now the claim follows from Lemma \ref{lb proj res}.
\end{proof}

\begin{theorem}\label{wt w}
Let $L_\lambda$ be any simple module in a weight $w$ block $\cB_{\rho,w}$ of $\cO_c(S_n)$. Then the vertex of $L_\lambda$ is $S_e^{\times w}$.
\end{theorem}
\begin{proof}
Lemma \ref{wt 1} implies that any simple module $\underline{L}$ 
 in the principal block $\cB_{\emptyset,1}^{\otimes w}$ of $\cO_c(S_e^{\times w})\simeq \cO_c(S_e)^{\otimes w}$ has vertex $S_e^{\times w}$. The proof of Theorem \ref{VerticesCherednik} showed that $L_\lambda\mid{^\cO\Ind}_{S_e^{\times w}}^{S_n} \underline{M} $ for some $\underline{M}$ in the principal block of $\cO_c(S_e^{\times w})$. We may always take $\underline{M}$ to be some simple module $\underline{L}$. Indeed, if $\underline{M}$ is not simple, then induce a non-split short exact sequence in which it appears in the middle, $L_\lambda$ is a direct summand of the middle term of the exact induced sequence, thus $L_\lambda$ is a summand of one of the outer terms, then do downwards induction on the composition length. The vertex of $L_\lambda$ is then the vertex of some simple module $\underline{L}\in\cB_{\emptyset,1}^{\otimes w}$, so by Lemma \ref{wt 1} it is $S_e^{\times w}$.
\end{proof}

\noindent
\textbf{Acknowledgments.}
The author is indebted to Olivier Dudas for helpful discussions and comments, and for the reference to Brou\'{e}'s generalization of Higman's criterion \cite[Theorem 6.8]{Broue}. Thanks also to Chris Bowman for answering a question about Littlewood-Richardson coefficients. The author was supported financially by MPIM Bonn at the time of writing. 
The author is also grateful to the anonymous referee for suggestions to clarify the writing.


\begin{thebibliography}{10}

\bibitem{Alperin}
J.L. Alperin.
\newblock{\em Local Representation Theory}.
\newblock Cambridge University Press, 1986.
  
\bibitem{BerestEtingofGinzburg}
Yuri Berest, Pavel Etingof, and Victor Ginzburg.
\newblock Finite-dimensional representations of rational Cherednik algebras.
\newblock {\em Int. Math. Res. Not. }2003, no. 19, 1053--1088. 
  
\bibitem{BezrukavnikovEtingof}
Roman Bezrukavnikov and Pavel Etingof. 
\newblock {Parabolic induction and restriction functors for rational Cherednik algebras}. 
\newblock {\em Selecta Math. (N.S.)} 14 (2009), no. 3-4, 397--425. 

\bibitem{Broue}
Michel Broue.
\newblock Higman's criterion revisited.
\newblock {\em Michigan Math. J.} 58 (2009).

\bibitem{CPS}
Edward Cline, Brian Parshall, and Leonard Scott.
\newblock Finite-dimensional algebras and highest weight categories.
\newblock {\em J. Reine Angew. Math.} 391 (1988), 85--99. 

\bibitem{DipperDu}
Richard Dipper and Jie Du. 
\newblock Trivial and alternating source modules of Hecke algebras.
\newblock{\em Proc. London Math. Soc.} 66 (1993) 479--506.

\bibitem{Du}
\newblock Jie Du.
\newblock The Green correspondence for the representations of Hecke algebras of type $A_{r-1}$.
\newblock{\em Transactions of the AMS}, vol. 329 no. 1 (Jan. 1992), 273--287.

\bibitem{BeijingLectures}
\newblock Olivier Dudas and Jean Michel.
\newblock Lectures on finite reductive groups and their representations.
\newblock \url{http://webusers.imj-prg.fr/~jean.michel/papiers/lectures_beijing_2015.pdf}

\bibitem{DipperJames}
\newblock Richard Dipper and Gordon James.
\newblock Blocks and idempotents of Hecke algebras of general linear groups.
\newblock{\em Proc. London Math. Soc.} (3) 54 (1987), no. 1, 57--82. 

\bibitem{GeckJacon}
\newblock Meinolf Geck and Nicolas Jacon.
\newblock{Representations of Hecke algebras at roots of unity.}
\newblock Algebra and Applications, 15.
\newblock{\em Springer-Verlag London, Ltd., London}, 2011.

\bibitem{GeckPfeiffer}
\newblock Meinolf Geck and G\"{o}tz Pfeiffer.
\newblock{\em Characters of finite Coxeter groups and Iwahori-Hecke algebras}.
\newblock Clarendon Press, Oxford, 2000.

\bibitem{GGOR}
\newblock Victor Ginzburg, Nicolas Guay, Eric Opdam, and Rapha\"{e}l Rouquier. On the category $\mathcal{O}$ for rational Cherednik algebras.
\newblock{\em Invent. Math.} 154 (2003), no. 3, 617--651.

\bibitem{GriffethJuteau}
\newblock Stephen Griffeth and Daniel Juteau.
\newblock {$W$-exponentials, Schur elements, and the support of the spherical representation of the rational Cherednik algebra}.
\newblock arXiv:1707.0819.

\bibitem{JamesKerber}
Gordon James and Adalbert Kerber.
\newblock {\em The Representation Theory of the Symmetric Group.}
\newblock Encyclopedia of Mathematics and its Applications, 16. Addison-Wesley Publishing Co., Reading, Mass., 1981. 

\bibitem{KMW}
Toshiro  Kuwabara,  Hyohe  Miyachi  and  Kentaro  Wada.
\newblock On the Mackey formula for cyclotomic Hecke algebras and categories $\cO$ of rational Cherednik algebras.
\newblock arXiv:1801.03761.

\bibitem{Linckelmann}
Markus Linckelmann.
\newblock {\em The Block Theory of Finite Group Algebras, Volume 1}. 
\newblock {\em London Mathematical Society Student Texts} {\textbf{91}}, Cambridge University Press, 2018.

\bibitem{Losev2013}
Ivan Losev.
\newblock{On isomorphisms of certain functors for Cherednik algebras.}
\newblock{\em Represent. Theory} 17 (2013), 247--262.

\bibitem{LosevKZ}
 Ivan Losev.
\newblock Finite-dimensional quotients of Hecke algebras.
\newblock {\em Algebra and Number Theory} 9:2 (2015). 493--502.

\bibitem{LosevSA}
 Ivan Losev and Seth Shelley-Abrahamson.
\newblock On refined filtration by supports for rational Cherednik categories $\mathcal{O}$.
\newblock{\em Selecta Math. (N.S.)} 24 (2018), no. 2, 1729--1804. 

\bibitem{MacLane}
Saunders Mac Lane.
\newblock{\em Categories for the Working Mathematician}.
\newblock Springer-Verlag, New York/Heidelberg/Berlin 1971.


\bibitem{Rouquier}
Rapha\"{e}l Rouquier.
\newblock{$q$-Schur algebras and complex reflection groups.}
\newblock {\em Mosc. Math. J.} 8 (2008), no. 1, 119--158, 184. 

\bibitem{Shan}
Peng Shan.
\newblock {Crystals of Fock spaces and cyclotomic rational double affine Hecke
  algebras}.
\newblock {\em Ann. Sci. \'Ec. Norm. Sup\'er.}, 44:147--182, 2011.

\bibitem{ShanVasserot}
Peng Shan and Eric Vasserot.
\newblock Heisenberg algebras and rational double affine Hecke algebras.
\newblock{\em J. Amer. Math. Soc.} 25 (2012), no. 4, 959--1031.

\bibitem{Whitley}
James R. Whitley.
\newblock Vertices for Iwahori-Hecke algebras and the Dipper-Du conjecture.
\newblock arXiv:1802.01406v2.

\bibitem{Wilcox}
Stewart Wilcox.
\newblock Supports of representations of the rational Cherednik algebra of type A.
\newblock {\em Adv. Math.} 314 (2017), 426--492.

\end{thebibliography}
\end{document}